\documentclass[12pt,reqno]{amsart}

\usepackage[margin=3cm]{geometry}
\title{Unusual spectral categories}

\author{Mar\'{i}a Jos\'{e} Arroyo Paniagua}
	\address[mariajose.mja@gmail.com]{Departamento de Matem\'{a}ticas,  Universidad Aut\'{o}noma Metropolitana - Iztapalapa, Mexico, M\'{e}xico}
	
\author{Alberto Facchini}
\address[facchini@math.unipd.it]{Dipartimento di Matematica ``Tullio Levi-Civita'', Universit\`{a} di Padova, Via Trieste 63, 35121 Padova, Italy}

\author{Marino Gran}
\address[marino.gran@uclouvain.be]{Universit\'{e} catholique de Louvain, Institut de Recherche en Math\'{e}matique et Physique, Chemin du Cyclotron 2, 1348 Louvain-la-Neuve, Belgique}
	
	\author{George Janelidze}
	\address[george.janelidze@uct.ac.za]{Department of Mathematics and Applied Mathematics, University of Cape Town, Rondebosch 7700, South Africa}

\usepackage{float}
\usepackage{tikz-cd}
\usepackage[all]{xy}
\usepackage{hyperref}
\usepackage{amsmath,amssymb}
\usepackage{amsthm}

\usepackage{fancyhdr}
\usepackage{tikz}
 \usetikzlibrary{positioning}
 \tikzset{mynode/.style={draw,circle,inner sep=1pt,outer sep=0pt}}
\usepackage{fancyhdr}
\newtheorem{theorem}{Theorem}[section]

\newtheorem{proposition}[theorem]{Proposition}

\newtheorem{condition}[theorem]{Condition}

\newtheorem{definition}[theorem]{Definition}

\newtheorem{remark} [theorem]{Remark}

\newtheorem{example}[theorem]{Example}

\begin{document}

 \begin{abstract}
		The paper is devoted to a kind of `very non-abelian' spectral categories. Under strong conditions on a category $\mathcal{X}$, we prove, among other things, that, for a given faithful localization $\mathcal{C}\to\mathcal{X}$, we have canonical equivalences $\mathrm{Spec}(\mathcal{C})\sim\mathcal{X}\sim(\mathrm{Category\,\,of\,\,injective\,\,objects\,\,in}\,\, \mathcal{C})$, and that $\mathcal{C}$ has natural injective envelopes.  
	\end{abstract}
	\thanks{The second author is partially supported by Ministero dell'Istruzione, dell'Universit\`a e della Ricerca (Progetto di ricerca di rilevante interesse nazionale ``Categories, Algebras: Ring-Theoretical and Homological Approaches (CARTHA)'') and Dipartimento di Mate\-ma\-tica ``Tullio Levi-Civita'' of Universit\`a di Padova (Research program DOR1828909 ``Anelli e categorie di moduli''). The fourth author was partially supported by South African NRF
}
\smallskip

	\maketitle
	





	\section{Introduction}
	
	\vspace{-.1in}
		Consider the passage 
	\usetikzlibrary{arrows,
		chains}
	
	\tikzset{
		decision/.style = {draw},
		line/.style = {draw},
		block/.style = {rectangle, draw,  text width=10 em, minimum height=10 mm,
			align=center}
	}
	\makeatletter
	\tikzset{suspend join/.code={\def
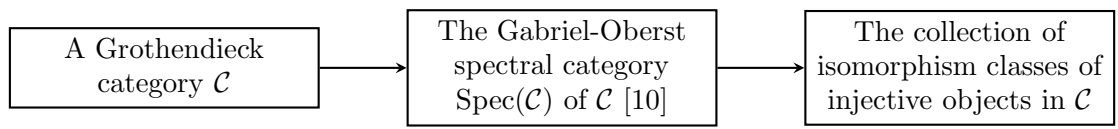
\tikz@after@path{}}}
	\makeatother
	
	\vspace{-.1in}
		\begin{figure}[htbp]\small
			\begin{tikzpicture}[thick,
			node distance = 0ex and 3em,
			start chain = A going right,
			every join/.style = {draw, -stealth, thick},
			block/.append style = {on chain=A, join}
			]
			\node [block]   {A Grothendieck category $\mathcal{C}$};
			\node [block]    {The Gabriel-Oberst spectral category $\mathrm{Spec}(\mathcal{C})$ of $\mathcal{C}$ \cite{[GO]}};
			\node [block]   {The collection of isomorphism classes of injective objects in $\mathcal{C}$};
			\end{tikzpicture}
		\end{figure}

\vspace{-.1in}
	\noindent and observe:
	
	\begin{itemize}
		\item [(i)] On the one hand $\mathrm{Spec}(\mathcal{C})$ has the same objects as $\mathcal{C}$, and on the other hand there is a canonical bijection between the collection of isomorphism classes of objects of $\mathrm{Spec}(\mathcal{C})$ and the collection of isomorphism classes of injective objects of $\mathcal{C}$ \cite{[GO]}.
		\item [(ii)] The canonical functor $P_{\mathcal{C}}:\mathcal{C}\to\mathrm{Spec}(\mathcal{C})$ has the property that, for objects $A$ and $B$ in $\mathcal{C}$, $P(A)$ is isomorphic to $P(B)$ if and only if $A$ and $B$ have isomorphic injective envelopes (=injective hulls). This was also shown in \cite{[GO]}.
		\item [(iii)] In spite of the property above, there is no pointed endofunctor $(I,\iota)$ of $\mathcal{C}$ such that, for an object $A$ in $\mathcal{C}$, $\iota_A:A\to I(A)$ is a monomorphism making $I(A)$ an injective envelope of $A$ (unless all objects of $\mathcal{C}$ are injective of course). We could express this fact by saying that $\mathcal{C}$ has no natural injective envelopes. As the title of \cite{[AHRT]} says, ``injective hulls are not natural''. We recalled more about this in \cite{[AFGJ]} with references to  \cite{[F]} and \cite{[GO]}.
	\end{itemize}
	The Gabriel\textendash Oberst construction of $\mathrm{Spec}(\mathcal{C})$, introduced in \cite{[GO]} for Grothendieck categories and considered in the setting of $G$-groups in \cite{[AF]}, was widely generalized in \cite{[AFGJ]}: the only assumption on $\mathcal{C}$ there was that $\mathcal{C}$ admits finite limits. It was achieved by inverting only pullback stable essential monomorphisms instead of inverting all essential monomorphisms; in the abelian case this does not change anything of course. In fact, even more generally, the construction given in \cite{[AFGJ]} involves a chosen class of monomorphisms in $\mathcal{C}$, but we will not go that far in the present paper.
	
	While the story of injective objects was not developed in \cite{[AFGJ]}, now we are interested in the following question: are there `strange', surely non-abelian, situations, in which $\mathrm{Spec}(\mathcal{C})$ is not equivalent to $\mathcal{C}$, and yet $\mathcal{C}$ has natural injective envelopes? Also note that the main result of \cite{[AHRT]} suggests not just to avoid abelian cases, but to make sure to have a lot of non-extremal monomorphisms.
	
	Well, here is what seems to be the best trivial example:
	
	\begin{example}
		Let $L$ be a $\wedge$-semilattice (with the largest element $1$), considered as a category. We observe:
		\begin{itemize}
			\item [(a)] Since every morphism in $L$ is a monomorphism, all extremal monomorphisms in $L$ are isomorphisms, and, at the same time, all morphisms in $L$ are pullback stable essential monomorphisms.
			\item [(b)] $L$ has exactly one injective object, namely the largest element $1$, and it is the unique injective envelope for each object. Since every diagram in $L$ commutes, this also means that $L$ has natural injective envelopes.
			\item [(c)] Since $\mathrm{Spec}(L)$ is constructed by inverting pullback stable essential monomorphisms in $L$, it is a trivial category; in particular, it is equivalent (even isomorphic) to the category of injective objects in~$L$. 
		\end{itemize}
	\end{example}   	 
	
	Does this trivial example suggest non-trivial ones?
Rather obviously, thinking, say, of the lattice of topologies on a given set, one should go to the theory of topological categories (see \cite{[B]} and references therein). Thinking this way our aim was first of all to find a fairly general non-abelian context for (not just having natural injective envelopes, but also) having the spectral category $\mathrm{Spec}(\mathcal{C})$ behaving much better than the original category $\mathcal{C}$ itself. And, of course, unlike the abelian case, where ``behaves much better'' means that every short exact sequence splits, what it should mean in the situation we aimed at was unclear to us. 
	
	The desired non-abelian context we found is described in this paper. It is fairly general indeed, although it is not too far from what is considered in \cite{[B]}. Specifically:
	\begin{itemize}
		\item Our first theorem (in Section 2) involves a new notion of generalized fibration (we call it ``weak fibration'').
		\item Then more results need further requirements, which brings us to faithful essential localizations in Section 4.
		\item We introduce extra conditions, under which our last theorems, 6.6 and 7.1, describe what turned out to be the `ideal behavior' of spectral categories; this means several good properties including the canonical equivalences $\mathrm{Spec}(\mathcal{C})\sim\mathcal{X}\sim(\mathrm{The\,\,category\,\,of\,\,injective\,\,objects\,\,in}\,\, \mathcal{C})$, for a given faithful localization $\mathcal{C}\to\mathcal{X}$, with the naturality of injective envelopes in $\mathcal{C}$, and dual properties.  
	\end{itemize}
		
	\textit{Throughout this paper} $\mathcal{C}$ \textit{and} $\mathcal{X}$ \textit{denote categories with finite limits and finite colimits (in fact we need only pullbacks and, considering `dual' situations, pushouts)}.

	\section{Preservation and reflection of essential monomorphisms}
	
	Recall that a monomorphism $m:S\to A$ in a given category is said to be an \textit{essential monomorphism} if a morphism $f:A\to B$ is a monomorphism whenever so is $fm$. From this very definition we immediately obtain:
	
	\begin{proposition}\label{ess-reflection}
		If a functor $F:\mathcal{C}\to\mathcal{X}$ preserves and reflects monomorphisms, then it reflects essential monomorphisms.\qed
	\end{proposition}
	
	The story of preservation of essential monomorphisms is a bit more complicated:
	
	\begin{definition}
		We will say that a functor $F:\mathcal{C}\to\mathcal{X}$ is:
		\begin{itemize}
			\item [(a)] a \textit{weak fibration} if, for every object $C$ in $\mathcal{C}$ and morphism $u:X\to F(C)$ in $\mathcal{X}$, there exist a morphism $f:U\to C$ in $\mathcal{C}$ and an epimorphism $\varphi:F(U)\to X$ in $\mathcal{X}$ with $F(f)=u\varphi$; we will also say that it is \textit{special} if $\varphi$ above can always be chosen to be an isomorphism;
			\item [(b)] a \textit{weak opfibration} if, for every object $C$ in $\mathcal{C}$ and a morphism $v:F(C)\to Y$, there exist a morphism $g:C\to V$ in $\mathcal{C}$ and a monomorphism $\psi:Y\to F(V)$ in $\mathcal{X}$ with $F(g)=\psi v$;
			\item [(c)] a \textit{weak bifibration} if it is a weak fibration and a weak opfibration at the same time; we will also say that it is \textit{special} if it is (also) a special weak fibration. 
		\end{itemize}
	\end{definition}
	
	\begin{example}
		Every Grothendieck fibration and, more generally, every Street fibration of categories (see \cite{[St]}) is obviously a weak fibration.
	\end{example}
	
	\begin{example}\label{sle}
		Let 
		\begin{equation*}
		(F,G,\eta,\varepsilon):\mathcal{C}\to\mathcal{X}
		\end{equation*}
		be an adjunction whose unit components $\eta_C:C\to GF(C)\,\,(C\in\mathcal{C})$ are pullback stable epimorphisms and whose counit components $\varepsilon_X:FG(X)\to X\,\,(X\in\mathcal{X})$ are epimorphisms. Then $F$ is a weak fibration. Indeed, for an object $C$ in $\mathcal{C}$ and a morphism $u:X\to F(C)$,  we can form the pullback
		\begin{equation*}\xymatrix{C\times_{GF(C)}G(X)\ar[d]_{\pi_1}\ar[r]^-{\pi_2}&G(X)\ar[d]^{G(u)}\\C\ar[r]_-{\eta_C}&GF(C)}
		\end{equation*}
		and then take $f:U\to C$ to be $\pi_1:C\times_{GF(C)}G(X)\to C$ and $\varphi:F(U)\to X$ to be the composite
		\begin{equation*}\xymatrix{F(C\times_{GF(C)}G(X))\ar[r]^-{F(\pi_2)}&FG(X) \ar[r]^-{\varepsilon_X}&X,}
		\end{equation*} 
		where $F(\pi_2)$ and $\varepsilon_X$ are epimorphisms by our assumptions. This example is motivated, of course, by the well-known connection between admissible/semi-left-exactness reflections and fibrations (see \cite{[BCGS]} and \cite{[CGJ]} and references therein for details). In particular, every semi-left-exact reflection (in the sense of \cite{[CHK]}) is a special weak fibration and every essential localization is a special weak bifibration. 
	\end{example}
	
	\begin{theorem}\label{weak-opfib}
		If $F:\mathcal{C}\to\mathcal{X}$ is a weak opfibration that preserves and reflects monomorphisms, then a morphism $m$ in $\mathcal{C}$ is an essential monomorphism if and only if the morphism $F(m)$ is an essential monomorphism in~$\mathcal{X}$.    
	\end{theorem}
	\begin{proof}
		Thanks to Proposition \ref{ess-reflection}, it suffices to prove that $F$ preserves essential monomorphisms. Let $F:\mathcal{C}\to\mathcal{X}$ be a weak opfibration that preserves and reflects monomorphisms, $m:S\to C$ an essential monomorphism in $\mathcal{C}$, and $v:F(C)\to X$ a morphism in $\mathcal{X}$ for which the composite $vF(m):F(S)\to X$ is a monomorphism in $\mathcal{X}$. We observe:
		\begin{itemize}
			\item Since $F$ is a weak opfibration, there exist a morphism $g$ in $\mathcal{C}$ and a monomorphism $\psi:Y\to F(V)$ in $\mathcal{X}$ with $F(g)=\psi v$.
			\item Since $F$ is faithful and $F(gm)=\psi vF(m)$ is a monomorphism in $\mathcal{X}$, $gm$ is a monomorphism in $\mathcal{C}$.
			\item Since $m$ is an essential monomorphism in $\mathcal{C}$, and $gm$ is a monomorphism in $\mathcal{C}$, $g$ is a monomorphism in $\mathcal{C}$.
			\item Since $g$ is a monomorphism in $\mathcal{C}$, and $F$ preserves monomorphisms, $v=F(g)$ is a monomorphism in $\mathcal{X}$.
		\end{itemize}
		This proves that $F(m)$ is an essential monomorphism in $\mathcal{X}$.  
	\end{proof}
		
	\section{Preservation and reflection of pullback stable essential monomorphisms}
	
	Using Proposition \ref{ess-reflection} we obtain:

	\begin{proposition}\label{pres-reflect}
		If a functor $F:\mathcal{C}\to\mathcal{X}$ preserves pullbacks and reflects monomorphisms, then it reflects pullback stable essential monomorphisms.\qed
	\end{proposition}
	
	And, again, the story of preservation is more complicated:
	
	\begin{theorem}\label{m ifff F(m)}
		If $F:\mathcal{C}\to\mathcal{X}$ is a special weak bifibration that preserves pullbacks and reflects monomorphisms, then a morphism $m$ in $\mathcal{C}$ is a pullback stable essential monomorphism if and only if the morphism $F(m)$ is a pullback stable essential monomorphism in $\mathcal{X}$.
	\end{theorem}
	
	\begin{proof}
		Thanks to Proposition \ref{pres-reflect}, it suffices to prove that $F$ preserves pullback stable essential monomorphisms. Let $m:S\to C$ be a pullback stable essential monomorphism in $\mathcal{C}$ and let
			\begin{equation*}\xymatrix{F(S)\times_{F(C)}X\ar[d]\ar[r]&X\ar[d]^u\\F(S)\ar[r]_-{F(m)}&F(C)}
			\end{equation*}
		be a pullback diagram in $\mathcal{X}$; we have to show that the pullback projection $F(S)\times_{F(C)}X\to X$ is an essential monomorphism in $\mathcal{X}$. For, since $F$ is a special weak fibration, we can choose a morphism $f:U\to C$ in $\mathcal{C}$ and an isomorphism $\varphi:F(U)\to X$ in $\mathcal{X}$ with $F(f)=u\varphi$, which allows us to argue as follows:
		\begin{itemize}
			\item Since $m$ is a pullback stable essential monomorphism and $F$ preserves pullbacks, the pullback projection $F(S)\times_{F(C)}F(U)\to F(U)$ of
			\begin{equation*}\xymatrix{F(S)\times_{F(C)}F(U)\ar[d]\ar[r]&F(U)\ar[d]^{F(f)}\\F(S)\ar[r]_-{F(m)}&F(C)}
			\end{equation*}
			is an essential monomorphism by Theorem \ref{weak-opfib}.			\item Since $\varphi$ is an isomorphism with $F(f)=u\varphi$, this implies that the pullback projection $F(S)\times_{F(C)}X\to X$ is also an essential monomorphism.			
		\end{itemize}
	\end{proof}
	
	\section{Involving essential localizations}
	
	Let us recall now, although essential localizations were already mentioned in Example \ref{sle}:
	
	\begin{definition}
		Let $\mathcal{C}$ be a category with finite limits. A \textit{localization} of $\mathcal{C}$ is an adjunction
		\begin{equation}
		(F,G,\eta,\varepsilon):\mathcal{C}\to\mathcal{X},
		\end{equation}
		in which $F:\mathcal{C}\to\mathcal{X}$ preserves finite limits while $G:\mathcal{X}\to\mathcal{C}$ is fully faithful. We will also say that the localization above is:
		\begin{itemize}
			\item [(a)] \textit{faithful}, if so is the functor $F$;
			\item [(b)] \textit{essential}, if the functor $F$ has a left adjoint.
		\end{itemize}   
	\end{definition}
	
	We will usually speak of just $F$ instead of $(F,G,\eta,\varepsilon):\mathcal{C}\to\mathcal{X}$. The following proposition is well known:
	
	\begin{proposition}
		For an arbitrary functor $F$, we have:
		\begin{itemize}
			\item [(a)] If $F=(F,G,\eta,\varepsilon)$ is an essential localization, then the left adjoint of $F$ is fully faithful.
			\item [(b)] If $F$ preserves pullbacks, then it preserves monomorphisms; in particular it is the case when $F$ is a localization. Consequently every essential localization preserves monomorphisms and epimorphisms.
			\item [(c)] If $F$ is faithful, then it reflects monomorphisms and epimorphisms.
		\end{itemize} 
	\end{proposition}

	\begin{remark}\label{Hopf}
		Localizations are rare in non-additive algebraic categories \cite{[CGJ]}.  An example of essential localization which is not faithful can be given in the semi-abelian category \cite{JMT} $\mathsf{Hopf}_{K, coc}$ of cocommutative Hopf algebras over an algebraically closed field $K$ with characteristic $0$ (see \cite{[GKV]} for more details). In this case we can take
		\begin{itemize}
			\item ${\mathcal C} = \mathsf{Hopf}_{K, coc}$;
			\item $\mathcal{X}=\mathsf{GrpHopf}_K$, the full subcategory of $\mathsf{Hopf}_{K, coc}$ whose objects are group-Hopf algebras, i.e., those Hopf algebras which are generated (as vector spaces) by group-like elements;
			\item $F \colon \mathsf{Hopf}_{K, coc} \rightarrow \mathsf{GrpHopf}_K$ is the functor sending a Hopf algebra to its group-Hopf subalgebra generated by group-like elements;
			\item the embedding functor $\mathsf{GrpHopf}_K \rightarrow  \mathsf{Hopf}_{K, coc}$ is both a left and right adjoint of $F$.
				\end{itemize}
				\end{remark} 
	 In the following we shall be interested in \emph{faithful} essential localizations: examples of such localizations are given here below in Example \ref{example-faithful} and in Remark \ref{Topo-functor}. An essential localization (1) is in fact a selfdual data having, together with (1), an adjunction $(H,F,\zeta,\theta):\mathcal{X}\to\mathcal{C}$ that determines a left adjoint of $F$. This gives, for each object $C$ in $\mathcal{C}$ and each object $X$ in $\mathcal{X}$, the diagrams
	\begin{equation}\label{split-extension}
	\xymatrix{HF(C)\ar[r]^-{\theta_C}&C\ar[r]^-{\eta_C}&GF(C),&FG(X)\ar[r]^-{\varepsilon_X}&X\ar[r]^-{\zeta_X}&FH(X).}
	\end{equation}
	As already mentioned in Example \ref{sle}, these assumptions make $F$ a special bifibration. Furthermore, having a left adjoint $F$ preserves all existing limits, and, since it is also required to be faithful, it preserves and reflects essential monomorphisms and pullback stable essential monomorphisms, by Theorems \ref{weak-opfib} and \ref{m ifff F(m)}, respectively. 
	
	\begin{remark}
		Since the functors $G$ and $H$ are fully faithful, the morphisms $\varepsilon_X$ and $\zeta_X$ above are always isomorphisms, which also implies that so are $F(\theta_C)$ and $F(\eta_C)$.  When $F$ is also faithful, it follows that $\theta_C$ and $\eta_C$ are bimorphisms, that is, they are monomorphisms and epimorphisms at the same time. In particular, if $\mathcal{C}$ is balanced, that is, every bimorphism in it is an isomorphism (or, informally, it satisfies the `equation' $\mathrm{mono}+\mathrm{epi}=\mathrm{iso}$), then $F$ is a category equivalence.   
	\end{remark} 
	
	\begin{example}\label{example-faithful}
		A paradigmatic example of a faithful essential localization is the following:
		\begin{itemize}
			\item $\mathcal{C}=\mathsf{Top}$, the category of topological spaces;
			\item $\mathcal{X}=\mathsf{Set}$, the category of sets;
			\item $F$ is the forgetful functor $\mathsf{Top}\to\mathsf{Set}$;
			\item $G$ and $H$ are the functors $\mathsf{Set}\to\mathsf{Top}$, which equip sets with the indiscrete and discrete topologies, respectively;
			\item all arrows involved in (2) are identity maps at the set level, that is, $F(\theta_C)=1_C=F(\eta_C)$ and $\varepsilon_X=1_X=\zeta_X$.
		\end{itemize}
	Similarly, when $\mathcal{C}=\mathsf{Grp(Top)}$ is the category of topological groups and $\mathcal{X}= \mathsf{Grp}$ the category of groups (identified with the category of indiscrete groups), the forgetful functor $F \colon \mathsf{Grp(Top)}\to\mathsf{Grp}$ has both a left and a right adjoint $\mathsf{Grp}\to\mathsf{Grp(Top)}$, namely the functors $G$ and $H$ which equip a group with the indiscrete and discrete topologies, respectively. The same remains true if one replaces the algebraic theory of groups with any algebraic theory in the sense of universal algebra.
	
	\end{example}
		\begin{remark}\label{Topo-functor}
		As suggested by Example \ref{example-faithful} we could require $FG=1_{\mathcal{X}}=FH$ and $\varepsilon$ and $\zeta$ to be identity natural transformations, which would bring us to the context of `concrete categories' in the sense of \cite{[AHS]}, and give many similar examples: see Section 8 in \cite{[AHS]}. In fact any topological functor as defined, say, in \cite{[H]}, \cite{[B]}, and \cite{[AHS]}, gives such an example, and there are many other examples constructed as various relational structures, not mentioned in these papers. Note, on the other hand, that the requirement above would not change much, except making our context more suitable for using the language of Grothendieck fibrations involving cartesian liftings. Concerning the relationship between localizations and fibrations see \cite{[BCGS]} and \cite{[CGJ]}.    
	\end{remark}
	
	A further convenient condition on the data above is:
	
	\begin{condition}\label{essential-stable}
		Every pullback stable essential monomorphism in $\mathcal{X}$ is an isomorphism.
	\end{condition}
	
	
	From Theorem \ref{m ifff F(m)} we obtain:
	
	\begin{theorem}\label{spectral-ess}
		Let $F:\mathcal{C}\to\mathcal{X}$ be a faithful essential localization satisfying Condition \ref{essential-stable}. Then a morphism $m$ in $\mathcal{C}$ is a pullback stable essential monomorphism if and only if the morphism $F(m)$ is an isomorphism in~$\mathcal{X}$.\qed
	\end{theorem}

	\section{The spectral category}
		
	The spectral category $\mathrm{Spec}(\mathcal{C},\mathcal{S})$ of $(\mathcal{C},\mathcal{S})$, where $\mathcal{S}$ is a class of monomorphisms in $\mathcal{C}$ satisfying suitable conditions, was defined in our previous paper \cite{[AFGJ]}; when $\mathcal{S}$ is the class of all monomorphisms in $\mathcal{C}$, we wrote $\mathrm{Spec}(\mathcal{C})$ instead of $\mathrm{Spec}(\mathcal{C},\mathcal{S})$. Recall that	
    \begin{equation*}
    \mathrm{Spec}(\mathcal{C})=\mathcal{C}[(\mathrm{St}(\mathrm{Mono}_E(\mathcal{C})))^{-1}]
    \end{equation*}   
	is the category of fractions of $\mathcal{C}$ for the class $\mathrm{St}(\mathrm{Mono}_E(\mathcal{C}))$ of pullback stable essential monomorphisms of $\mathcal{C}$. However, under the assumptions of Section 4, including Condition \ref{essential-stable}, Theorem \ref{spectral-ess} tells us that $\mathrm{St}(\mathrm{Mono}_E(\mathcal{C}))$ is nothing but the class of all morphisms in $\mathcal{C}$ whose $F$-images are isomorphisms. After that Proposition 1.3 of Chapter I in \cite{[GZ]} gives
	
	\begin{theorem}\label{spectral}
		Let $F:\mathcal{C}\to\mathcal{X}$ be a faithful essential localization satisfying Condition \ref{essential-stable}. Then $F$ factors uniquely through the canonical functor $\mathcal{C}\to\mathrm{Spec}(\mathcal{C})$, and the resulting functor $\bar{F}:\mathrm{Spec}(\mathcal{C})\to\mathcal{X}$ is a category equivalence.
	\end{theorem}
	
	\section{Bimorphisms and duality}
	
Let us now compare Condition \ref{essential-stable} with the following ones:
	
	\begin{condition}\label{split}
		Every monomorphism in $\mathcal{X}$ is split.
	\end{condition}
	
	\begin{condition}\label{balanced}
		$\mathcal{X}$ is balanced.
	\end{condition}
	
	\begin{proposition}\label{comparison}
		Conditions \ref{essential-stable}, \ref{split} and \ref{balanced} are related to each other as follows:
		\begin{itemize}
			\item [(a)] $\ref{split} \Rightarrow( \ref{essential-stable}\,\&\,\ref{balanced})$;
			\item [(b)] none of the implications $ \ref{essential-stable} \Rightarrow \ref{balanced}$, $\ref{balanced} \Rightarrow  \ref{essential-stable}$, and $( \ref{essential-stable}\,\&\,\ref{balanced})\Rightarrow \ref{split}$ holds in general.
		\end{itemize}
	\end{proposition}
	\begin{proof}
		(a) is obvious. To prove (b) we have the following counter-examples:
		\begin{itemize}
			\item Let $\mathcal{X}$ be the variety of universal algebras with a single unary operation $\omega$ satisfying the identity $\omega(x)=\omega(y)$. This makes $\omega$ a constant operation, and, informally, $\mathcal{X}$ consists of pointed sets and the empty set put together. Or, formally, $\mathcal{X}$ can be seen as the category of pointed sets together with a freely added new initial object. It is easy to see that $\mathcal{X}$ satisfies Condition \ref{essential-stable}, but not Condition \ref{balanced}.
			\item The category of abelian groups (where the class of essential monomorphisms is pullback stable) satisfies Condition \ref{balanced}, but not Condition \ref{essential-stable}.
			\item Condition \ref{split} does not hold in the category $\mathsf{Set}$ of sets, although both conditions $\ref{essential-stable}$ and $\ref{balanced}$ do.
			Note that $0\to 1$ is an essential (non-split and non-pullback-stable-essential) monomorphism in $\mathsf{Set}$. The fact that every essential monomorphism in $\mathsf{Set}$ is either an isomorphism or of the form $0\to 1$ is a trivial but nice observation we found in \cite{[AHS]}.
		\end{itemize} 
	\end{proof}
	\begin{remark}
		It would be more complicated to compare either essential monomorphisms with bimorphisms, or pullback stable essential monomorphisms with pullback stable bimorphisms, but fortunately we did not need that. 
	\end{remark}
	Theorem \ref{spectral-ess} was the best we could do with pullback stable essential monomorphisms under Condition \ref{essential-stable}, while we could do nothing under Condition \ref{balanced} alone. Therefore Proposition \ref{comparison} suggests to consider the conjunction of Conditions \ref{essential-stable} and \ref{balanced}, and Theorem \ref{spectral-ess} immediately gives:

	\begin{theorem}\label{bimorphism}
		Let $F:\mathcal{C}\to\mathcal{X}$ be a faithful essential localization satisfying Conditions \ref{essential-stable} and \ref{balanced}. Then a morphism $m$ in $\mathcal{C}$ is a pullback stable essential monomorphism if and only if it is a bimorphism.
	\end{theorem}
	
	Next, note that while the notion of faithful essential localization and Conditon \ref{balanced} are self-dual, Condition  \ref{essential-stable} is not. But why not dualizing it? Theorem \ref{bimorphism}, together with Theorem \ref{spectral}, will immediately give:
	\begin{theorem}
		Let $F:\mathcal{C}\to\mathcal{X}$ be a faithful essential localization satisfying Conditions \ref{essential-stable}, its dual condition, and Condition \ref{balanced}. Then the classes of pullback stable essential monomorphisms in $\mathcal{C}$ and in $\mathcal{C}^\mathrm{op}$ coincide with each other and with the class of bimorphisms in $\mathcal{C}$; in particular, the dual of the canonical functor $\mathcal{C}\to\mathrm{Spec}(\mathcal{C})$ can be identified with the canonical functor $\mathcal{C}^\mathrm{op}\to\mathrm{Spec}(\mathcal{C}^\mathrm{op})$, and
		\begin{equation*}
		\mathrm{Spec}(\mathcal{C})\approx\mathcal{X}\approx\mathrm{Spec}(\mathcal{C}^\mathrm{op})^\mathrm{op}.\qed
		\end{equation*}
	
	\end{theorem}
	
	There are many examples for $\mathcal{X}=\mathsf{Set}$, since in this case Conditions \ref{essential-stable}, its dual condition, and Condition \ref{balanced} always hold. In particular, we could choose $\mathcal{C}$ to be any category topological over $\mathsf{Set}$ (=a category equipped with a topological functor from it to $\mathsf{Set}$; see Remark \ref{Topo-functor}), or the dual of such a category.
	
	\section{Involving injective and projective objects}
	
	In this section, we will use the notation introduced in Section 4, right after Remark \ref{Hopf}, in order to formulate a single theorem, omitting various preliminary results.
	
	\begin{theorem}\label{7.1}
		Let $F:\mathcal{C}\to\mathcal{X}$ be a faithful essential localization. Then, in the notation of Section 4, we have:
		\begin{itemize}
			\item [(a)] If Condition \ref{split} holds, then, for every object $C$ in $\mathcal{C}$, the morphism $\eta_C:C\to GF(C)$ is an injective envelope of $C$, and, in particular, $C$ is an injective object if and only if $\eta_C$ is an isomorphism. In this case the categories $\mathrm{Spec}(\mathcal{C})$ and $\mathcal{X}$ are canonically equivalent to the category of injective objects in $\mathcal{C}$.
			\item [(b)] If the dual of Condition \ref{split} holds, then, for every object $C$ in $\mathcal{C}$, the morphism $\theta_C:HF(C)\to C$ is a projective cover of $C$, and, in particular, $C$ is a projective object if and only if $\theta_C$ is an isomorphism. In this case the categories $\mathrm{Spec}(\mathcal{C^\mathrm{op}})$ and $\mathcal{X}^\mathrm{op}$ are canonically equivalent to the dual category of projective objects in $\mathcal{C}$.  
		\end{itemize}
		
	\end{theorem}
	\begin{proof}
		Since (a) and (b) are dual to each other, we will only prove (a). Since Condition \ref{split} implies Conditions \ref{essential-stable} and \ref{balanced} (by Proposition \ref{comparison}), we can use all the previous theorems.
		
		\textit{Claim 1. For every object $X$ in $\mathcal{X}$, $G(X)$ is an injective object in $\mathcal{C}$.}
		
		Proving this is a standard argument:
		
		Given a monomorphism $f:A\to B$ in $\mathcal{C}$, we have to prove that the map 
		\begin{equation*}
		\mathrm{hom}_\mathcal{C}(f,G(X)):\mathrm{hom}_\mathcal{C}(B,G(X))\to \mathrm{hom}_\mathcal{C}(A,G(X))
		\end{equation*}
		is surjective. But the surjectivity of this map is equivalent to the surjectivity of the map
		\begin{equation*}
		\mathrm{hom}_\mathcal{C}(F(f),X):\mathrm{hom}_\mathcal{C}(F(B),X)\to \mathrm{hom}_\mathcal{C}(F(A),X),
		\end{equation*}
		which follows from the fact $F$ preserves monomorphisms and Condition \ref{split}.
		
		\textit{Claim 2. For every object $C$ in $\mathcal{C}$, the morphism $\eta_C:C\to GF(C)$ is an essential monomorphism in $\mathcal{C}$.}
		
		This follows from Theorem \ref{spectral-ess}.
		
		From Claims 1 and 2 we conclude that, for every object $C$ in $\mathcal{C}$, the morphism $\eta_C:C\to GF(C)$ is an injective envelope of $C$, and, in particular, that $C$ is an injective object if and only if $\eta_C$ is an isomorphism. The last assertion of (a) follows from Theorem \ref{spectral}, and the fact the category of objects $C$ in $\mathcal{C}$ for which $\eta_C$ is an isomorphism is equivalent to the category~$\mathcal{C}$.		
	\end{proof}
	
	Note that Theorem \ref{7.1}(a) does not apply to $\mathcal{C}=\mathsf{Top}$, since $\mathsf{Set}$ does not satisfy Condition \ref{split}, but applies to the category of pointed topological spaces, and \ref{7.1}(b) applies to $\mathsf{Top}$. The same can be said about all categories that are topological over $\mathsf{Set}$. 
	


\begin{thebibliography}{9}

\bibitem{[AHRT]} \textsc{J. Ad\'amek, H. Herrlich, J. Rosick\'y and W. Tholen,} \emph{Injective hulls are not natural}, Algebra	Universalis ~\textbf{48} (2002) pp.~374--388.

\bibitem{[AHS]} \textsc{J. Ad\'amek, H. Herrlich and G.E. Strecker,} \textit{Abstract and concrete categories. The joy of cats,} Pure and Applied Mathematics, A Wiley-Interscience Publication, John Wiley \& Sons, Inc., New York, 1990.

\bibitem{[AF]}\textsc{M.J. Arroyo Paniagua and A. Facchini,} 
	\textit{Category of $G$-groups and its spectral category,} Comm. Algebra ~\textbf{45-4} (2017), pp.~1696--1710.

				
	\bibitem{[AFGJ]}\textsc{M. J. Arroyo Paniagua, A. Facchini,  M. Gran and G. Janelidze,} \textit{What is a spectral category?,}
accepted for publication in ``Advances
in Rings, Modules and Factorizations'', A. Facchini, M.
Fontana, A. Geroldinger and B. Olberding Eds., Springer,
New York, 2020. arXiv:1903.10034.
				
			
	\bibitem{[BCGS]} \textsc{F. Borceux, M. M. Clementino, M. Gran and L. Sousa,} {Protolocalisations of homological categories, } J.~Pure Appl.~Algebra~\textbf{212} (2008), pp.~1898--1927.
				
	\bibitem{[B]}\textsc{G. C. L. Br\"ummer,}  \textit{Topological categories, } Topology and its Applications ~\textbf{18} (1984), pp.~27--41.
				
	\bibitem{[CHK]} \textsc{C. Cassidy, M. H\'ebert and G. M. Kelly,} \textit{Reflective subcategories, localizations and factorization systems,} J. Austral. Math. Soc. Ser. A ~\textbf{38-3} (1985), pp.~287--329.
				
	\bibitem{[CGJ]} \textsc{M. M. Clementino, M. Gran and G. Janelidze,} \textit{Some remarks on protolocalizations and protoadditive reflections,} J. Algebra and Applications~\textbf{17-11} (2018), 1850207, 16 pp., DOI: 10.1142/S0219498818502079 .
	
	\bibitem{[F]} \textsc{A. Facchini,}  \textit{Injective modules, spectral categories, and applications}, In ``Noncommutative
	rings, group rings, diagram algebras and their applications'', S. K. Jain and S. Parvathi Eds.,
	Contemp. Math.~\textbf{456}, Amer. Math. Soc., Providence, RI (2008), pp.~1--17.

								
	\bibitem{[GO]}\textsc{P. Gabriel and U. Oberst,} \textit{Spektralkategorien und reguläre Ringe im von-Neumannschen} Sinn, Math. Z. ~\textbf{92} (1966), pp.~389--395.
				
	\bibitem{[GZ]} \textsc{P. Gabriel and M. Zisman,} \textit{Calculus of fractions and homotopy theory,} Springer, Berlin, 1967.
	
	\bibitem{[GKV]}  \textsc{M. Gran, G. Kadjo and J. Vercruysse,} \textit{Split extension classifiers in the category of cocommutative algebras,} Bull. Belg. Math. Soc. Simon Stevin ~\textbf{25} (2018), pp.~355--382. 
	
	\bibitem{[H]}\textsc{ H. Herrlich,}\textit{ Topological functors,} General Topology and its Applications ~\textbf{4-2} (1974), pp.~125--142.
	
	\bibitem{JMT} \textsc{G. Janelidze, L. M\'arki and W. Tholen}, \textit{Semi-abelian categories,} {J.~Pure Appl.~Algebra} \textbf{168} (2002), pp.~367--386.

	\bibitem{[St]}\textsc{R. H. Street,} \textit{Fibrations in bicategories,} Cahiers de Topologie et G\'eom\'etrie Diff\'erentielles Cat\'egoriques ~\textbf{21-2} (1980), pp.~111--160.
\end{thebibliography}
	\end{document}